\documentclass[11 pt]{article}
\usepackage{color,amsthm,amsmath,array,amssymb,amscd,amsfonts,latexsym, url,wasysym,mathtools}
\usepackage{bbold}
\usepackage{graphicx}
\usepackage[colorlinks,allcolors=blue]{hyperref}
\usepackage{xcolor}
\usepackage{lmodern}
\usepackage{tikz}
\usepackage{varwidth}
\usetikzlibrary{calc}
\usetikzlibrary{decorations.pathreplacing}

\usepackage{ytableau}
\usepackage{mathtools}
\usepackage{listings}
\usepackage{longtable}
\usepackage{pythonhighlight}
\usepackage{cite}
\usepackage{rotating}
\newcommand*\circled[1]{\tikz[baseline=(char.base)]{\node[shape=circle,draw,inner sep=2pt] (char) {#1};}}

\newtheorem{theo}{Theorem}[section]

\newtheorem{lemm}[theo]{Lemma}

\newtheorem{rema}[theo]{Remark}

\newtheorem{ex}[theo]{Example}

\title{\bf A Murnaghan--Nakayama Rule for Grothendieck Polynomials of Grassmannian Type}
\author{Duc-Khanh Nguyen, Dang Tuan Hiep, Tran Ha Son, Do Le Hai Thuy}

\date{}

\newfont{\gothic}{eufb10}

\begin{document}
\maketitle
\begin{abstract} 
We consider the Grothendieck polynomials appearing in the $K$-theory of Grassmannians, which are analogs of Schur polynomials. This paper aims to establish a version of the Murnaghan--Nakayama rule for Grothendieck polynomials of the Grassmannian type. This rule allows us to express the product of a Grothendieck polynomial with a power-sum symmetric polynomial into a linear combination of other Grothendieck polynomials.  
\end{abstract}

\section{Introduction}

The $K$-theory of flag varieties was studied by Kostant and Kumar \cite{kostant1987t}, and by Demazure \cite{demazure1974desingularisation}. Lascoux and Schutzenberger introduced the Grothendieck polynomials as representatives for the structure sheaves of the Schubert varieties of a flag variety \cite{gherardelli2006invariant, lascoux1982structure, lascoux2007anneau}. For any permutation $w \in \bigcup_{m \geq 1}S_m$, the Grothendieck polynomial   $\mathfrak{G}_w:= \mathfrak{G}_w(x_1, x_2, \dots)$ is defined by isobaric divided difference operators. Fomin and Kirilov studied
 combinatorics of these polynomials in \cite{fomin1996yang, fomin1994grothendieck}.\\

Let $s_\lambda$ be the Schur function associated with a partition $\lambda$, and $p_k$ be the power-sum symmetric functions of degree $k$ \cite{macdonald1991notes}. The classical Murnaghan--Nakayama rule describes the decomposition of the product $s_\lambda p_k$ into a sum of Schur functions \cite{macdonald1991notes} as follows. We have  
\begin{equation*}
s_\lambda p_k = \sum_{\mu}  (-1)^{r(\mu/\lambda)+1} s_\mu,    
\end{equation*}
where the sum runs over all partitions $\mu$ such that $\mu/\lambda$ is a ribbon of size $k$ and $r(\mu/\lambda)$ is the number of rows of skew shape $\mu/\lambda$.\\

The classical Murnaghan--Nakayama rule plays an important role in the representation theory of symmetric groups. It gave a formula for the character table \cite{nakayama1940someII,nakayama1940someI,murnaghan1937characters}. Because of the classical story, many extensions and generalizations of the classical Murnaghan--Nakayama rule were studied. Indeed, a version for non-commutative symmetric functions is given by Fomin and Green in \cite{fomingreen1998} (it led to formulas for characters of representations associated with stable Schubert and Grothendieck polynomials). A skew version and its generalization of multiplication with quantum power-sum function are given by \cite{konvalinka2012skew,assafmcnamara2011}. A version for non-commutative Schur functions can be found in \cite{tewari2016murnaghan}. A plethystic version is given by \cite{wildon2016combinatorial}. A version for loop Schur functions is given by \cite{ross2014loop} (it provides a fundamental step in the orbifold Gromov--Witten/Donaldson--Thomas correspondence in \cite{ross2013gerby}). A version in the cohomology of an affine Grassmannian can be found in \cite{bandlow2011murnaghan}. An extended version of Schubert polynomials and the quantum cohomology of Grassmannians can be found in \cite{morrison2018two}.\\

In this paper, we restrict our attention to the simplest complex flag variety: the Grassmann variety of $n$ dimensional subspaces of $\mathbb{C}^{n+m}$. The operator-definition of $\mathfrak{G}_{w}$ reduces in the special case of Grassmannian permutations of descent $n$ to symmetric polynomials in $n$ variables indexed by partitions $\lambda$ with at most $n$ parts, commonly relabeled as $G_\lambda$ and called Grassmannian Grothendieck polynomials. These $G_\lambda$ are $K$-theoretic analogs of Schur functions, representing the $K$-theory of the Grassmannian variety of $n$ dimensional subspaces of $\mathbb{C}^{n+m}$. Like the Schur functions, they are given by a number of formulas (see Sect. \ref{Grothendieck_poly_Gr}). For example
\begin{equation*}
    G_\lambda(x_1,\dots,x_n) = \frac{\det (x_i^{\lambda_j + n - j}(1 +\beta x_i)^{j-1})_{n \times n}}{ \prod_{1 \leq i <j \leq n} (x_i-x_j)},
\end{equation*}
where $\beta$ is a formal parameter. Recall that, if $\beta=0$, then $G_\lambda$ is identified with the Schur function $s_\lambda(x_1, \dots, x_n)$. The products of $G_\lambda$ with other special symmetric polynomials $e_k,h_k$ are mentioned in \cite{lenart2000combinatorial}, in which Lenart studied the Pieri rules of the Grassmannian Grothendieck polynomials. Our work on the product $G_\lambda p_k$ can be considered as a $K$-theoretic version of the classical Murnaghan--Nakayama rule.\\

Let $\lambda$ and $\mu$ be partitions of length at most $n$ and $\lambda \leq \mu$. Let $|\mu/\lambda|, c(\mu/\lambda), r(\mu/\lambda)$ be the size, number of columns, number of rows of the skew shape $\mu/\lambda$. We say two boxes in a skew shape are {adjacent} whenever they share an edge, and we say a skew shape $\mu/\lambda$ is {connected} whenever every pair of its boxes is connected by a sequence of adjacent boxes. A {ribbon} is a connected skew shape with no $2\times2$ square. When $\mu/\lambda$ is connected, the {maximal ribbon along the northwest border of $\mu/\lambda$} is the ribbon $\nu/\lambda$ of size as max as possible, contained in $\mu/\lambda$.\\

The main result of this paper is stated as follows. 
\begin{theo}\label{main}
For any partition $\lambda$ of length at most $n$ and $k \in \mathbb{Z}_{>0}$, we have
\begin{equation*}\label{main1-equation}
G_\lambda p_k = \sum_{\mu}(-\beta)^{|\mu/\lambda|-k}(-1)^{k-c(\mu/\lambda)}\binom{r(\mu/\lambda)-1}{k-c(\mu/\lambda)}G_{\mu},
\end{equation*}
where the sum runs over all partitions $\mu$ of length at most $n$, $\mu \geq \lambda$ such that $c(\mu/\lambda) \leq k$, $\mu/\lambda$ is connected and the maximal ribbon along its northwest border has size at least $k$. 
\end{theo}

This paper is organized as follows. In Sect. \ref{pre} we recall the basic knowledge related to symmetric polynomials, partitions, diagrams, binary tableaux, Grothendieck polynomials of Grassmannian type. In Sect. \ref{proof} we prove our main result.\\

\textbf{Acknowledgments:} This work is supported by the Ministry of Education and Training, Vietnam, under project code B2022-CTT-02: ``Study some combinatorial models in Representation Theory", 2022-2023 (Decision No.1323/QD-BGDDT, May 19, 2022). Khanh would like to express his sincere gratitude for the Visiting Fellowship supported by MathCoRe and Prof. Petra Schwer at Otto-von-Guericke-Universität Magdeburg. He would also like to thank Prof. Cristian Lenart for his strong encouragement and for taking his time to read the manuscript and give valuable comments. Hiep would like to thank Prof. Takeshi Ikeda for introducing Grothendieck polynomials and explaining their importance in the study of $K$-theory of Grassmannians. Thuy was partially supported by the Vietnam Academy of Science and Technology under Grant Number DLTE00.04/23-24. 

\section{Preliminaries}\label{pre}
\subsection{Symmetric Polynomials}
A polynomial $f(x_1,\dots,x_n)$ in $n$ variables is said to be {\em symmetric} if for all permutations $\sigma \in S_n$, then we have
\[f(x_{\sigma(1)},\dots,x_{\sigma(n)}) = f(x_1,\dots,x_n).\]
There are fundamental symmetric polynomials: The $k$-th {\em elementary symmetric polynomial}
\[e_k = \sum_{1\leq i_1<\cdots<i_k\leq n}x_{i_1}\cdots x_{i_k},\]
the $k$-th {\em complete homogeneous symmetric polynomial}
\[h_k = \sum_{1\leq i_1\leq \cdots\leq i_k\leq n}x_{i_1}\cdots x_{i_k},\]
and the $k$-th {\em power-sum symmetric polynomial}
\[p_k = \sum_{i=1}^n x_i^k.\]
Let $k$ be a positive integer. The following formula is key to the proof of Theorem \ref{main}.  
\begin{lemm}
\begin{equation}\label{p}
p_k = \sum_{i = 0}^{k-1}(-1)^{i}(k-i)e_ih_{k-i}.
\end{equation}
\end{lemm}
\begin{proof}
The proof of the equality (\ref{p}) is as follows. We consider the following generating functions
\[H(t) = \sum_{k\geq 0}h_kt^k = \prod_{i=1}^n\frac{1}{1-x_it} ,\]
\[E(t) = \sum_{k \geq 0}e_kt^k = \prod_{i=1}^n(1+x_it) ,\]
\[P(t) = \sum_{k\geq 1}p_kt^{k-1} = \sum_{i=1}^n\frac{x_i}{1 - x_it} .\]
By (2.6), (2.10) in \cite{macdonald1998symmetric}, we have
\[P(t) = H'(t)/H(t) = H'(t)E(-t).\]
By comparing the coefficients of $t^{k-1}$ in both sides of the identity, we get the conclusion. 
\end{proof}

\subsection{Partitions, Diagrams and Binary Tableaux}
A non-negative integer sequence $\lambda = (\lambda_1, \lambda_2, \dots)$ is called a {\em partition} if $\lambda_1 \geq \lambda_2 \geq \dots$. If $\lambda=(\lambda_1,\lambda_2,\dots,\lambda_l)$ with $\lambda_l >0$ and $\sum\limits_{i=1}^l \lambda_i = m$, we write $l(\lambda)=l$, $|\lambda| = m$. We call $l(\lambda)$ the {\em length}, and $|\lambda|$ the {\em size} of the partition $\lambda$. Each partition $\lambda$ is presented by a {\em Young diagram} that is a collection of boxes such that: The leftmost boxes of each row are in a column, and the numbers of boxes from top row to bottom row are $\lambda_1,\lambda_2,\dots$, respectively. 

\begin{ex}
\normalfont
The Young diagram associated to the partition $\lambda =(3,3,2,1,0,0)$ is
$$
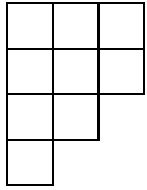
\begin{ytableau}
\,&\,&\,\\
\,&\,&\,\\
\,&\,\\
\\
\end{ytableau}$$
\end{ex}

Let $\lambda=(\lambda_1,\lambda_2, \dots)$ and $\nu = (\nu_1,\nu_2,\dots)$ be partitions. We define the sum of two partitions by $\lambda + \nu = (\lambda_1 + \nu_1, \lambda_2+\nu_2, \dots)$. For a non-negative integer $n$, we denote $\mathcal{P}_n$ the set of all partition of length at most $n$. Let $(1^n)$ be the $n$-tuple partition $(1,\dots,1)$, and $\lambda = (\lambda_1, \dots, \lambda_n) \in \mathcal{P}_n$, then we have $\lambda+(1^n)=(\lambda_1+1, \dots, \lambda_n+1)$.

Let $\lambda =(\lambda_1,\lambda_2,\dots)$ and $\mu=(\mu_1,\mu_2,\dots)$ be two partitions. We say that $\lambda$ is {\em smaller} than $\mu$ if and only if $\lambda_i \leq \mu_i$ for all $i$, and we write $\lambda \leq \mu$. In this case, we define the {\em skew Young diagram} $\mu/ \lambda$ as the result of removing boxes in the Young diagram $\lambda$ from the Young diagram $\mu$. We write $|\mu/\lambda|=|\mu|-|\lambda|$ for the size, and $r(\mu/\lambda), c(\mu/\lambda)$ for the number of rows, columns of the skew Young diagram $\mu/\lambda$, respectively. We say two boxes in a skew shape are {\em adjacent} whenever they share an edge, and we say a skew shape $\mu/\lambda$ is {\em connected} whenever every pair of its boxes is connected by a sequence of adjacent boxes. A {\em ribbon} is a connected skew shape with no $2\times2$ square. The {\em maximal ribbon along the northwest border of a connected skew Young diagram $\mu/\lambda$} is the ribbon $\nu/\lambda$ of size as max as possible, contained in $\mu/\lambda$. A {\em binary tableau} $T$ of skew shape $\mu/\lambda$ is a result of filling the skew Young diagram $\mu/\lambda$ by the alphabet $\{0,1\}$ such that the entry in the bottom of each column is $1$. A binary tableau $T$ is said to have {\em content} $\alpha(T)=(\alpha_0,\alpha_1)$ if $\alpha_i = \alpha_i(T)$ is the number of entries $i$ in $T$. We write $sh(T)$ for the shape of the tableau $T$.

\begin{ex}\label{skew_diagram} 
\normalfont
We consider partitions in $\mathcal{P}_6$: $\lambda = (3,3,2,1,0,0)$ and $\mu = (4,3,3,3,1,1)$. Then $\mu \geq \lambda$ and the skew diagram $\mu/\lambda$ has $r(\mu/\lambda)=5$ rows, $c(\mu/\lambda) = 4$ columns. In this case $\mu/\lambda$ is not connected and is not a ribbon. However, it contains ribbons of size $1,2,3$, for example
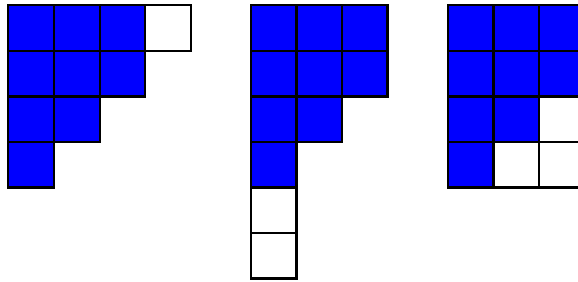
\[\begin{ytableau}
*(gray) & *(gray) & *(gray) & *(white) \\
*(gray) & *(gray) & *(gray)   \\
*(gray) & *(gray)   \\
*(gray) 
\end{ytableau} \quad\quad \begin{ytableau}
*(gray) & *(gray) & *(gray) \\
*(gray) & *(gray) & *(gray)   \\
*(gray) & *(gray)  \\
*(gray)  \\
*(white) \\
*(white) 
\end{ytableau}  \quad\quad
\begin{ytableau}
*(gray) & *(gray) & *(gray) \\
*(gray) & *(gray) & *(gray)   \\
*(gray) & *(gray) & *(white)   \\
*(gray) & *(white) & *(white)
\end{ytableau}
\]
The following tableau $T$ is a binary tableau of skew shape $sh(T)=\mu/\lambda$.

$$
\begin{tikzpicture}[inner sep=0in,outer sep=0in]
\node (n) {\begin{varwidth}{4cm}{
\begin{ytableau}
*(gray) & *(gray) & *(gray) & *(white) 1\\
*(gray) & *(gray) & *(gray)   \\
*(gray) & *(gray) & *(white) 0  \\
*(gray) & *(white) 1 & *(white) 1\\
*(white) 1 \\
*(white) 1
\end{ytableau} 
}\end{varwidth}};
\draw[very thick,black] (n.south west)--([xshift=0.6cm,yshift=0cm]n.south west)--([xshift=0.6cm,yshift=1.2cm]n.south west)--([xshift=1.8cm,yshift=1.2cm]n.south west)--([xshift=1.8cm,yshift=2.4cm]n.south west)--([xshift=1.8cm,yshift=3cm]n.south west)--([xshift=2.4cm,yshift=3cm]n.south west)--([xshift=2.4cm,yshift=3.6cm]n.south west)--([xshift=1.8cm,yshift=3.6cm]n.south west)--([xshift=1.8cm,yshift=2.4cm]n.south west)--([xshift=1.2cm,yshift=2.4cm]n.south west)--([xshift=1.2cm,yshift=1.8cm]n.south west)--([xshift=0.6cm,yshift=1.8cm]n.south west)--([xshift=0.6cm,yshift=1.2cm]n.south west)--([xshift=0cm,yshift=1.2cm]n.south west)--([xshift=0cm,yshift=0cm]n.south west)
;
\end{tikzpicture}
$$
Here the diagram in gray means the Young diagram $\lambda$ removed from the Young diagram $\mu$. We bold-outline the skew shape $\mu/\lambda$. The content of the binary tableau $T$ is $\alpha(T)=(1,5)$. Other example, let $\nu = (2,1,0,0,0,0)$, then $\mu/\nu$ is connected and is not a union of ribbons. The following is a binary tableau of skew shape $\mu/\nu$, with content $(5,7)$. 
$$
\begin{tikzpicture}[inner sep=0in,outer sep=0in]
\node (n) {\begin{varwidth}{4cm}{
\begin{ytableau}
*(gray) & *(gray) & *(white)1 & *(white) 1\\
*(gray) & *(white) 1 & *(white)0  \\
*(white)0 & *(white)0 & *(white) 0  \\
*(white)0 & *(white) 1 & *(white) 1\\
*(white) 1 \\
*(white) 1
\end{ytableau} 
}\end{varwidth}};
\draw[very thick,black] (n.south west)--([xshift=0.6cm,yshift=0cm]n.south west)--([xshift=0.6cm,yshift=1.2cm]n.south west)--([xshift=1.8cm,yshift=1.2cm]n.south west)--([xshift=1.8cm,yshift=3cm]n.south west)--([xshift=2.4cm,yshift=3cm]n.south west)--([xshift=2.4cm,yshift=3.6cm]n.south west)--([xshift=1.2cm,yshift=3.6cm]n.south west)--([xshift=1.2cm,yshift=3cm]n.south west)--([xshift=0.6cm,yshift=3cm]n.south west)--([xshift=0.6cm,yshift=2.4cm]n.south west)--([xshift=0cm,yshift=2.4cm]n.south west)--([xshift=0cm,yshift=0cm]n.south west)
;
\end{tikzpicture}
$$

\end{ex}

\subsection{Grothendieck Polynomials of Grassmannian Type}\label{Grothendieck_poly_Gr}
In this paper, we restrict our attention to the simplest complex flag variety: the Grassmann variety of $n$ dimensional subspaces of $\mathbb{C}^{n+m}$. The Grothendieck polynomials in this case are indexed by Grassmannian permutations \cite{buch2002littlewood}. Namely, let $\lambda=(\lambda_1, \dots, \lambda_n)$ be a partition of length at most $n$. The Grassmannian permutation $w_\lambda$ of descent $n$ is defined by $w_\lambda(i) = i +\lambda_{n+1-i}$ for $i \in [1,n]$  and $w_\lambda(i) < w_{\lambda}(i+1)$ for all $i \ne n$. Set $G_\lambda = \mathfrak{G}_{w_\lambda}$. There are several new formulas for $G_\lambda$, for example, in the terms of set-valued tableaux  \cite{buch2002littlewood} or Jacobi--Trudy identity \cite{kirillov2016some}. We here recall the Weyl identity given by Ikeda and Naruse \cite{ikeda2014proof,ikedanaruse2013}: 
\begin{equation*}\label{G2}G_\lambda(x_1,\dots,x_n) = \frac{\det (x_i^{\lambda_j + n - j}(1 +\beta x_i)^{j-1})_{n \times n}}{ \prod_{1 \leq i <j \leq n} (x_i-x_j)}.
\end{equation*}
They are polynomial representatives of Schubert classes of the $K$-theory of Grassmann varieties of $n$ dimensional subspaces of $\mathbb{C}^{n+m}$. 
\begin{ex}
\normalfont
For $n = 3$ and $\lambda = (2,1,0)$, we have 
\begin{align*}G_{(2,1,0)}(x_1,x_2,0) &= \frac{\begin{vmatrix}
x_1^4 & x_1^2(1+\beta x_1) & (1+\beta x_1)^2\\
x_2^4 & x_2^2(1+\beta x_2) & (1+\beta x_2)^2\\
0 & 0 & 1
\end{vmatrix}}{(x_1 - x_2)x_1 x_2}\\
&= x_1^2x_2 + x_1x_2^2 + \beta x_1^2x_2^2.
\end{align*}
We can check that $G_{(2,1,0)}(x_1,x_2,0) = G_{(2,1)}(x_1,x_2)$. 
\end{ex}
\begin{rema}
\normalfont
For $\beta = 0$, the Grothendieck polynomial $G_\lambda(x_1,\dots,x_n)$ coincides with the Schur polynomial $s_\lambda(x_1,\dots,x_n)$. 
\end{rema}
\section{Proof of the Main Theorem} \label{proof}
We only need to focus on the case $\beta \ne 0$. Indeed, if $\beta = 0$, then $G_\lambda$ is the Schur function associated with partition $\lambda$. The Murnaghan--Nakayama rule for Schur polynomials is very well known \cite{macdonald1991notes}. When $\beta \neq 0$, set
\[\widetilde{G}_\lambda(x_1,\dots,x_n) = \beta^{|\lambda|}G_\lambda\left(\frac{x_1}{\beta},\dots,\frac{x_n}{\beta}\right).\]
Theorem \ref{main} can be reduced to the following theorem.
\begin{theo}\label{main1}
For any partition $\lambda \in \mathcal{P}_n$ and $k \in \mathbb{Z}_{>0}$, we have
\begin{equation}\label{main10}
\widetilde{G}_\lambda p_k = \sum_{\mu}(-1)^{|\mu/\lambda|-c(\mu/\lambda)}\binom{r(\mu/\lambda)-1}{k-c(\mu/\lambda)}\widetilde{G}_{\mu},
\end{equation}
where the sum runs over all partitions $\mu \in \mathcal P_n$, $\mu \geq \lambda$ such that $c(\mu/\lambda) \leq k$, $\mu/\lambda$ is connected and the maximal ribbon along its northwest border has size at least $k$. 
\end{theo}
Before going to the proof, we need to restate the following lemma. It was key to obtain the Pieri rules for Grothendieck polynomials of Grassmannian type \cite{lenart2000combinatorial}.
\begin{lemm}[Theorem 3.2, \cite{lenart2000combinatorial}]\label{Grothendieck_Pieri}
For any partition $\lambda \in \mathcal{P}_n$ and $k \in \mathbb{N}$, we have
\begin{equation}\label{lenart1}
\widetilde{G}_\lambda e_k = \sum_{T} (-1)^{\alpha_0(T)} \widetilde{G}_\mu,
\end{equation}
\begin{equation}\label{lenart2}
\widetilde{G}_\lambda h_k = \sum_{T} (-1)^{\alpha_0(T)} \widetilde{G}_\mu.
\end{equation}
The first sum runs over all binary tableaux $T$ of shape $\mu/\lambda$ with $\mu \in \mathcal{P}_n$, $\lambda \leq \mu \leq \lambda + (1^n)$, $\alpha_1(T) = k$. The second sum runs over all binary tableaux $T$ of shape $\mu/\lambda$ with $\mu \in \mathcal{P}_n$ , $\lambda \leq \mu$, $\alpha_1(T) = k$, no two $1$'s in the same column.  
\end{lemm}
\begin{proof}[Proof of Theorem \ref{main1}]
By equalities (\ref{p}), (\ref{lenart1}), (\ref{lenart2}), we have  
\begin{align}
    \widetilde{G}_\lambda p_k &= \sum_{i=0}^{k-1}(-1)^i(k-i) \widetilde{G}_\lambda e_ih_{k-i}\\
    &=\sum_{i=0}^{k-1}(-1)^i(k-i)\sum_{T_1} (-1)^{\alpha_0(T_1)} \widetilde{G}_\nu h_{k-i} \label{0_last_equality}\\
    \intertext{where $T_1$ has shape $\nu/\lambda$ with $\nu \in \mathcal{P}_n$, $\lambda \leq \nu \leq \lambda+(1^n)$, $\alpha_1(T_1)=i$,}
    &= \sum_{i=0}^{k-1}(-1)^i(k-i)\sum_{T_1} (-1)^{\alpha_0(T_1)}\sum_{T_2} (-1)^{\alpha_0(T_2)} \widetilde{G}_\mu \label{1_last_equality}\\
    \intertext{where $T_2$ has shape $\mu/\nu$ with $\mu \in \mathcal{P}_n$, $\nu \leq \mu$, $\alpha_1(T_2) = k-i$, no two $1$'s in the same column,}
    &=\sum_{\mu}\sum_{T = T_1\sqcup T_2} (-1)^{|\mu/\lambda| - \alpha_1(T_2)}\alpha_1(T_2)\widetilde{G}_\mu, \label{last_equality}
\end{align}
because $$i = \alpha_1(T_1), k-i = \alpha_1(T_2),$$ and $$\alpha_0(T_1)+\alpha_1(T_1)+\alpha_0(T_2)+\alpha_1(T_2)=|\mu/\lambda|.$$
In (\ref{last_equality}), the sum runs over binary tableaux $T = T_1 \sqcup T_2$ of shape $\mu/\lambda$, with $\mu \in \mathcal{P}_n$, $\mu \geq \lambda$, $\alpha_1(T) = k$, where $T_1$ and $T_2$ are filled according to (\ref{0_last_equality}) and (\ref{1_last_equality}). Fix such a shape $\mu$ containing $\lambda$, we are going to determine the form of $\mu$ and the coefficient of $\widetilde{G}_\mu$ appearing in the decomposition of $\widetilde{G}_\lambda p_k$. \\

\underline{First step:} Construct all tableaux $T$ mentioned in (\ref{last_equality}) of given skew shape $\mu/\lambda$ such that the numbers of entries with value $1$ in $T_2$ is a fixed number $j$. We proceed as follows.
\begin{itemize}
    \item First, we label all boxes in the bottom of each column of the skew diagram $\mu/\lambda$ by $1$. Let $\mathcal{B}$ be the set of boxes \fbox{1} we have created.  
    \item Now, we will choose a subset of boxes in $\mathcal{B}$ and set it as the set of boxes in the bottom of $T_2$, say $\mathcal{B}_2$. In fact, we cannot choose such a subset arbitrarily because its complement in $\mathcal{B}$, say $\mathcal{B}_1$, will be a subset of boxes in the bottom of $T_1$. Hence, it must satisfy a strict condition that the boxes in $\mathcal{B}_1$ are located in the skew diagram $(\lambda + (1^n))/\lambda$. Therefore, to choose a subset $\mathcal{B}_2$, we should start by choosing $\mathcal{B}_1$. Fixing a number of entries with value $1$ in $T_2$, say $\alpha_1(T_2)=j$, we have
    \begin{itemize}
        \item The cardinality of $\mathcal{B}_1$ is $c(\mu/\lambda)-j$. 
        \item The elements in $\mathcal{B}_1$ are chosen arbitrarily from $\gamma :=\mathcal{B}\cap (\lambda+(1^n))/\lambda$.
    \end{itemize}
    Hence, for a fixed $\alpha_1(T_2)=j$, the number of choices of $\mathcal{B}_2$ is equal to the number of choices of $\mathcal{B}_1$ and it is  
    \begin{equation}\label{choice_step2}
        \binom{|\gamma|}{c(\mu/\lambda)-j}.
    \end{equation}
    Since $\mathcal{B}_2= \mathcal{B}\setminus \mathcal{B}_1$, we have \begin{equation}\label{range_j}
         j=|\mathcal{B}_2| \in [|\mathcal{B}\setminus \mathcal{B}_1|, |\mathcal{B}|]= [c(\mu/\lambda)- |\gamma|,c(\mu/\lambda)].
    \end{equation}
    \item Now, the last step to construct tableau $T$ is locating remaining entries with value $1$ of $T$ which are not in the bottom $\mathcal{B}$ in the skew diagram $(\lambda+(1^n)/\lambda) \cap (\mu/\lambda)$. We have
    \begin{itemize}
        \item The number of remaining entries with value $1$ is $k-c(\mu/\lambda)$.
        \item Such entries with value $1$ are chosen arbitrarily from $\eta := (\lambda+(1^n)/\lambda) \cap (\mu/\lambda) {\setminus} \gamma$.
    \end{itemize}
    Hence, the number of choices of this step is 
    \begin{equation}\label{choice_step3}
        \binom{|\eta|}{k-c(\mu/\lambda)}.
    \end{equation}
\end{itemize}
Therefore, we have described a way to construct tableaux $T$ of given skew shape $\mu/\lambda$ such that the number of entries with value $1$ in $T_2$ is a fixed number $j$. \\

\underline{Second step:} Substitute (\ref{choice_step2}), (\ref{range_j}), (\ref{choice_step3}) to (\ref{last_equality}) and simplify it. We have:
\begin{equation}\label{equality_by_j}
    \widetilde{G}_\lambda p_k =\sum_{\mu}  \sum_{j=c(\mu/\lambda)-|\gamma|}^{c(\mu/\lambda)} (-1)^{|\mu/\lambda| - j}j \binom{|\gamma|}{c(\mu/\lambda)-j} \binom{|\eta|}{k-c(\mu/\lambda)} \widetilde{G}_\mu.
\end{equation}
We note that the binomial coefficient 
\[\binom{|\eta|}{k-c(\mu/\lambda)}\]
depends only on $\lambda,\mu$ and $k$. Thus, in order to simplify the coefficient of $\widetilde{G}_\mu$, we only need to determine the sum
\[\sum_{j=c(\mu/\lambda)-|\gamma|}^{c(\mu/\lambda)} (-1)^{|\mu/\lambda| - j}j \binom{|\gamma|}{c(\mu/\lambda)-j}.\]
Since $k >0$, then $|\gamma| \geq 1$. We prove the following lemma.
\begin{lemm}
The sum 
\begin{equation}\label{sum}
\sum_{j=c(\mu/\lambda)-|\gamma|}^{c(\mu/\lambda)} (-1)^{c(\mu/\lambda) - j}j \binom{|\gamma|}{c(\mu/\lambda)-j}
\end{equation}
is equal to $0$ if $|\gamma| > 1$ and $1$ if $|\gamma| =1$.
\end{lemm}
\begin{proof}
First, we consider the following identity:
\begin{equation}\label{1-x}
    (1-x)^m = \sum\limits_{i=0}^m (-x)^i \binom{m}{i}.
\end{equation}
When $m\geq 1$, $x=1$, from (\ref{1-x}), we get
\begin{equation}\label{m>=1,x=1}
    0 = \sum\limits_{i=0}^m(-1)^i \binom{m}{i}. 
\end{equation}
Differentiating both sides of (\ref{1-x}), we get
\begin{equation}\label{diff}
    m(1-x)^{m-1} = \sum\limits_{i=0}^m i(-x)^{i-1} \binom{m}{i}.
\end{equation}
When $m>1$, $x=1$, from (\ref{diff}), we get
\begin{equation}\label{diff_m>1,x=1}
    0 = \sum\limits_{i=0}^m (-1)^i i \binom{m}{i}.
\end{equation}

Now, we use the equalities above to prove the lemma. Set $i = c(\mu/\lambda)-j$ and $c = c(\mu/\lambda)$. Then, (\ref{sum}) can be rewritten as 
\[\sum_{i = 0}^{|\gamma|}(-1)^{i}(c-i)\binom{|\gamma|}{i} = c \sum_{i=0}^{|\gamma|} (-1)^i\binom{|\gamma|}{i} - \sum_{i=0}^{|\gamma|}(-1)^i i \binom{|\gamma|}{i}.\]
Since $|\gamma|\geq 1$, then by (\ref{m>=1,x=1}), we get
\[\sum_{i=0}^{|\gamma|} (-1)^i\binom{|\gamma|}{i} = 0.\]
If $|\gamma| > 1$, then by (\ref{diff_m>1,x=1}), we get
\[\sum_{i=0}^{|\gamma|}(-1)^i i \binom{|\gamma|}{i} = 0.\]
If $|\gamma| = 1$, then 
\[\sum_{i=0}^{|\gamma|}(-1)^i i \binom{|\gamma|}{i} = \sum_{i=0}^{1}(-1)^i i \binom{1}{i} = -1.\]
We obtain the result as desired. 
\end{proof}
Now, we consider two cases.
\begin{itemize}
    \item If $|\gamma|=1$, then $\mu/\lambda$ is connected (by definition of $\gamma$, the cardinality of $\gamma$ counts the number of connected components of $\mu/\lambda$). The reader may check that connectedness implies $|(\lambda+(1^n)/\lambda) \cap (\mu/\lambda)|=r(\mu/\lambda)$. Therefore, $|\eta| = r(\mu/\lambda)-1$. The coefficient of $\widetilde{G}_\mu$ in (\ref{equality_by_j}) is 
    \begin{equation*}
        (-1)^{|\mu/\lambda|-c(\mu/\lambda)}\binom{r(\mu/\lambda)-1}{k - c(\mu/\lambda)}.
    \end{equation*}
    \item If $|\gamma|>1$, the coefficient of $\widetilde{G}_\mu$ in (\ref{equality_by_j}) is $0$. 
\end{itemize}
When $|\gamma|=1$, the conditions that $\mu/\lambda$ is the shape of a tableau $T=T_1 \sqcup T_2$, where $T_1$, $T_2$ are of the form in (\ref{lenart1}), (\ref{lenart2}) and that $\alpha_1(T)=k$ are equivalent to the conditions $c(\mu/\lambda) \leq k$ (entries with value $1$ in bottom $\mathcal{B}$ is a part of all entries with value $1$ of $T$) and $k-c(\mu/\lambda) \leq r(\mu/\lambda)-1$ (entries with value $1$ not in bottom $\mathcal{B}$ can be filled into $\eta$). The last inequality condition can be rewritten as
\begin{equation}\label{ribbon_k_containing}
    k \leq c(\mu/\lambda) + r(\mu/\lambda)-1.
\end{equation}
Since $\mu/\lambda$ is connected, the right-hand side of (\ref{ribbon_k_containing}) counts the size of the maximal ribbon contained in the skew shape $\mu/\lambda$ along its northwest border. Hence, the conditions of $\mu$ such that $\widetilde{G}_\mu$ appears in the decomposition of $\widetilde{G}_\lambda p_k$ are: $\mu \in \mathcal{P}_n, \mu \geq \lambda$ such that $c(\mu/\lambda) \leq k$, $\mu/\lambda$ is connected and the maximal ribbon along its northwest border has size at least $k$.
\end{proof}
The example below visualizes the first step, constructing tableaux $T$, in the proof of Theorem \ref{main1}. 
\begin{ex} 
\normalfont
We continue Example \ref{skew_diagram}. In the picture below, $\mathcal{B}$ is the set of four boxes \fbox{\color{blue}{1}}, and the skew diagram $(\lambda+(1^n))/\lambda$ is colored in yellow. Therefore, $\gamma$ is the set of three yellow boxes with blue entries with value $1$ inside. 
$$
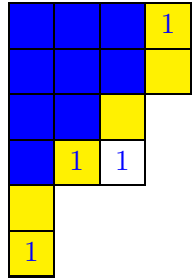
\begin{tikzpicture}[inner sep=0in,outer sep=0in]
\node (n) {\begin{varwidth}{4cm}{
\begin{ytableau}
*(gray) & *(gray) & *(gray) & *(yellow) \color{blue}{1}\\
*(gray) & *(gray) & *(gray) & *(yellow) \\
*(gray) & *(gray) & *(yellow)  \\
*(gray) & *(yellow) \color{blue}{1} & *(white) \color{blue}{1}\\
*(yellow)  \\
*(yellow) \color{blue}{1}
\end{ytableau} 
}\end{varwidth}};
\draw[very thick,black] (n.south west)--([xshift=0.6cm,yshift=0cm]n.south west)--([xshift=0.6cm,yshift=1.2cm]n.south west)--([xshift=1.8cm,yshift=1.2cm]n.south west)--([xshift=1.8cm,yshift=2.4cm]n.south west)--([xshift=1.8cm,yshift=3cm]n.south west)--([xshift=2.4cm,yshift=3cm]n.south west)--([xshift=2.4cm,yshift=3.6cm]n.south west)--([xshift=1.8cm,yshift=3.6cm]n.south west)--([xshift=1.8cm,yshift=2.4cm]n.south west)--([xshift=1.2cm,yshift=2.4cm]n.south west)--([xshift=1.2cm,yshift=1.8cm]n.south west)--([xshift=0.6cm,yshift=1.8cm]n.south west)--([xshift=0.6cm,yshift=1.2cm]n.south west)--([xshift=0cm,yshift=1.2cm]n.south west)--([xshift=0cm,yshift=0cm]n.south west)
;
\end{tikzpicture}
$$
The range of the number of entries with value $1$ in $T_2$ is $j \in [1,4]$. If we fix $j=2$, then there are three choices of $\mathcal{B}_1$ (also $\mathcal{B}_2$) as in the picture below (the entries with value $1$ in $\mathcal{B}_1$ are circled). 
$$
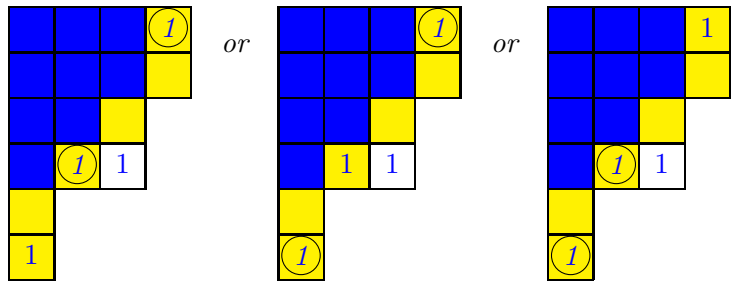
\begin{tikzpicture}[inner sep=0in,outer sep=0in]
\node (n) {\begin{varwidth}{4cm}{
\begin{ytableau}
*(gray) & *(gray) & *(gray) & *(yellow) \circled{\color{blue}{1}}\\
*(gray) & *(gray) & *(gray) & *(yellow) \\
*(gray) & *(gray) & *(yellow)  \\
*(gray) & *(yellow) \circled{\color{blue}{1}} & *(white) \color{blue}{1}\\
*(yellow)  \\
*(yellow) \color{blue}{1}
\end{ytableau} }\end{varwidth}};
\draw[very thick,black] (n.south west)--([xshift=0.6cm,yshift=0cm]n.south west)--([xshift=0.6cm,yshift=1.2cm]n.south west)--([xshift=1.8cm,yshift=1.2cm]n.south west)--([xshift=1.8cm,yshift=2.4cm]n.south west)--([xshift=1.8cm,yshift=3cm]n.south west)--([xshift=2.4cm,yshift=3cm]n.south west)--([xshift=2.4cm,yshift=3.6cm]n.south west)--([xshift=1.8cm,yshift=3.6cm]n.south west)--([xshift=1.8cm,yshift=2.4cm]n.south west)--([xshift=1.2cm,yshift=2.4cm]n.south west)--([xshift=1.2cm,yshift=1.8cm]n.south west)--([xshift=0.6cm,yshift=1.8cm]n.south west)--([xshift=0.6cm,yshift=1.2cm]n.south west)--([xshift=0cm,yshift=1.2cm]n.south west)--([xshift=0cm,yshift=0cm]n.south west)
;
\end{tikzpicture}
\quad\quad\quad
\begin{tikzpicture}[inner sep=0in,outer sep=0in]
\node (n) {\begin{varwidth}{4cm}{
\begin{ytableau}
*(gray) & *(gray) & *(gray) & *(yellow) \circled{\color{blue}{1}}\\
*(gray) & *(gray) & *(gray) & *(yellow) \\
*(gray) & *(gray) & *(yellow)  \\
*(gray) & *(yellow) \color{blue}{1} & *(white) \color{blue}{1}\\
*(yellow)  \\
*(yellow) \circled{\color{blue}{1}}
\end{ytableau} 
}\end{varwidth}};
\draw[very thick,black] (n.south west)--([xshift=0.6cm,yshift=0cm]n.south west)--([xshift=0.6cm,yshift=1.2cm]n.south west)--([xshift=1.8cm,yshift=1.2cm]n.south west)--([xshift=1.8cm,yshift=2.4cm]n.south west)--([xshift=1.8cm,yshift=3cm]n.south west)--([xshift=2.4cm,yshift=3cm]n.south west)--([xshift=2.4cm,yshift=3.6cm]n.south west)--([xshift=1.8cm,yshift=3.6cm]n.south west)--([xshift=1.8cm,yshift=2.4cm]n.south west)--([xshift=1.2cm,yshift=2.4cm]n.south west)--([xshift=1.2cm,yshift=1.8cm]n.south west)--([xshift=0.6cm,yshift=1.8cm]n.south west)--([xshift=0.6cm,yshift=1.2cm]n.south west)--([xshift=0cm,yshift=1.2cm]n.south west)--([xshift=0cm,yshift=0cm]n.south west)
;
\end{tikzpicture}
\quad\quad\quad
\begin{tikzpicture}[inner sep=0in,outer sep=0in]
\node (n) {\begin{varwidth}{4cm}{
\begin{ytableau}
*(gray) & *(gray) & *(gray) & *(yellow) \color{blue}{1}\\
*(gray) & *(gray) & *(gray) & *(yellow) \\
*(gray) & *(gray) & *(yellow)  \\
*(gray) & *(yellow) \circled{\color{blue}{1}} & *(white) \color{blue}{1}\\
*(yellow)  \\
*(yellow) \circled{\color{blue}{1}}
\end{ytableau}
}\end{varwidth}};
\draw[very thick,black] (n.south west)--([xshift=0.6cm,yshift=0cm]n.south west)--([xshift=0.6cm,yshift=1.2cm]n.south west)--([xshift=1.8cm,yshift=1.2cm]n.south west)--([xshift=1.8cm,yshift=2.4cm]n.south west)--([xshift=1.8cm,yshift=3cm]n.south west)--([xshift=2.4cm,yshift=3cm]n.south west)--([xshift=2.4cm,yshift=3.6cm]n.south west)--([xshift=1.8cm,yshift=3.6cm]n.south west)--([xshift=1.8cm,yshift=2.4cm]n.south west)--([xshift=1.2cm,yshift=2.4cm]n.south west)--([xshift=1.2cm,yshift=1.8cm]n.south west)--([xshift=0.6cm,yshift=1.8cm]n.south west)--([xshift=0.6cm,yshift=1.2cm]n.south west)--([xshift=0cm,yshift=1.2cm]n.south west)--([xshift=0cm,yshift=0cm]n.south west)
;
\end{tikzpicture}
$$

The skew shape $\eta$ contains two boxes where we put $*$ inside in the picture as follows.
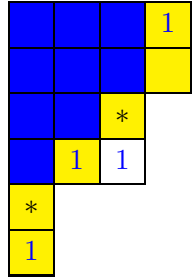
$$
\begin{tikzpicture}[inner sep=0in,outer sep=0in]
\node (n) {\begin{varwidth}{4cm}{
\begin{ytableau}
*(gray) & *(gray) & *(gray) & *(yellow) \color{blue}{1}\\
*(gray) & *(gray) & *(gray) & *(yellow) \\
*(gray) & *(gray) & *(yellow)*  \\
*(gray) & *(yellow) \color{blue}{1} & *(white) \color{blue}{1}\\
*(yellow)*  \\
*(yellow) \color{blue}{1}
\end{ytableau}
}\end{varwidth}};
\draw[very thick,black] (n.south west)--([xshift=0.6cm,yshift=0cm]n.south west)--([xshift=0.6cm,yshift=1.2cm]n.south west)--([xshift=1.8cm,yshift=1.2cm]n.south west)--([xshift=1.8cm,yshift=2.4cm]n.south west)--([xshift=1.8cm,yshift=3cm]n.south west)--([xshift=2.4cm,yshift=3cm]n.south west)--([xshift=2.4cm,yshift=3.6cm]n.south west)--([xshift=1.8cm,yshift=3.6cm]n.south west)--([xshift=1.8cm,yshift=2.4cm]n.south west)--([xshift=1.2cm,yshift=2.4cm]n.south west)--([xshift=1.2cm,yshift=1.8cm]n.south west)--([xshift=0.6cm,yshift=1.8cm]n.south west)--([xshift=0.6cm,yshift=1.2cm]n.south west)--([xshift=0cm,yshift=1.2cm]n.south west)--([xshift=0cm,yshift=0cm]n.south west)
;
\end{tikzpicture}
$$
If $k=5$, then we just need to put only one remaining entry $1$ arbitrarily to the boxes marked by $*$. The remaining boxes of $\mu/\lambda$ are labeled by $0$. For example, if we fix the first choice of $\mathcal{B}_1$ in the picture above ($j=2$), we have two tableaux below (empty yellow boxes are not counted in tableaux).
$$
\begin{tikzpicture}[inner sep=0in,outer sep=0in]
\node (n) {\begin{varwidth}{4cm}{
\begin{ytableau}
*(gray) & *(gray) & *(gray) & *(yellow) \circled{\color{blue}{1}}\\
*(gray) & *(gray) & *(gray) & *(yellow) \\
*(gray) & *(gray) & *(yellow) 1 \\
*(gray) & *(yellow) \circled{\color{blue}{1}} & *(white) \color{blue}{1}\\
*(yellow) 0  \\
*(yellow) \color{blue}{1}
\end{ytableau}
}\end{varwidth}};
\draw[very thick,black] (n.south west)--([xshift=0.6cm,yshift=0cm]n.south west)--([xshift=0.6cm,yshift=1.2cm]n.south west)--([xshift=1.8cm,yshift=1.2cm]n.south west)--([xshift=1.8cm,yshift=2.4cm]n.south west)--([xshift=1.8cm,yshift=3cm]n.south west)--([xshift=2.4cm,yshift=3cm]n.south west)--([xshift=2.4cm,yshift=3.6cm]n.south west)--([xshift=1.8cm,yshift=3.6cm]n.south west)--([xshift=1.8cm,yshift=2.4cm]n.south west)--([xshift=1.2cm,yshift=2.4cm]n.south west)--([xshift=1.2cm,yshift=1.8cm]n.south west)--([xshift=0.6cm,yshift=1.8cm]n.south west)--([xshift=0.6cm,yshift=1.2cm]n.south west)--([xshift=0cm,yshift=1.2cm]n.south west)--([xshift=0cm,yshift=0cm]n.south west)
;
\end{tikzpicture}
\quad \quad \quad
\begin{tikzpicture}[inner sep=0in,outer sep=0in]
\node (n) {\begin{varwidth}{4cm}{
\begin{ytableau}
*(gray) & *(gray) & *(gray) & *(yellow) \circled{\color{blue}{1}}\\
*(gray) & *(gray) & *(gray) & *(yellow) \\
*(gray) & *(gray) & *(yellow) 0 \\
*(gray) & *(yellow) \circled{\color{blue}{1}} & *(white) \color{blue}{1}\\
*(yellow) 1 \\
*(yellow) \color{blue}{1}
\end{ytableau}
}\end{varwidth}};
\draw[very thick,black] (n.south west)--([xshift=0.6cm,yshift=0cm]n.south west)--([xshift=0.6cm,yshift=1.2cm]n.south west)--([xshift=1.8cm,yshift=1.2cm]n.south west)--([xshift=1.8cm,yshift=2.4cm]n.south west)--([xshift=1.8cm,yshift=3cm]n.south west)--([xshift=2.4cm,yshift=3cm]n.south west)--([xshift=2.4cm,yshift=3.6cm]n.south west)--([xshift=1.8cm,yshift=3.6cm]n.south west)--([xshift=1.8cm,yshift=2.4cm]n.south west)--([xshift=1.2cm,yshift=2.4cm]n.south west)--([xshift=1.2cm,yshift=1.8cm]n.south west)--([xshift=0.6cm,yshift=1.8cm]n.south west)--([xshift=0.6cm,yshift=1.2cm]n.south west)--([xshift=0cm,yshift=1.2cm]n.south west)--([xshift=0cm,yshift=0cm]n.south west)
;
\end{tikzpicture}
$$

\end{ex}

\bibliography{references}{}
\bibliographystyle{alpha}

\noindent Institut für Algebra und Geometrie, Otto-von-Guericke-Universität Magdeburg, Germany\\
E-mail: \href{khanh.mathematic@gmail.com}{khanh.mathematic@gmail.com} \\

\noindent Faculty of Mathematics and Computer Science, University of Da Lat, 1 Phu Dong Thien Vuong, Da Lat City, Lam Dong, Vietnam\\
E-mail: \href{hiepdt@dlu.edu.vn}{hiepdt@dlu.edu.vn} \\

\noindent Decentralized Applied Crypto Laboratory, University of Science, VNU-HCMC, Ho Chi Minh City, Vietnam \\
E-mail: \href{tranhason1705@gmail.com}{tranhason1705@gmail.com} \\

\noindent Institute of Mathematics, Vietnam Academy of Science and Technology,
18 Hoang Quoc Viet Road, Cau Giay District, Hanoi, Vietnam\\
E-mail: \href{cbl.dolehaithuy@gmail.com}{cbl.dolehaithuy@gmail.com}

\end{document}